\documentclass[notitlepage,11pt,reqno]{amsart}
\usepackage[foot]{amsaddr}
\usepackage[numbers,sort]{natbib}

\usepackage{amssymb,nicefrac,bm,upgreek,mathtools,verbatim}
\usepackage[mathscr]{eucal}
\usepackage{dsfont}
\usepackage{calrsfs}
\DeclareMathAlphabet{\pazocal}{OMS}{zplm}{m}{n}
\usepackage{color}
\definecolor{dmagenta}{rgb}{.4,.1,.5}
\definecolor{dblue}{rgb}{.0,.0,.5}
\definecolor{mblue}{rgb}{.0,.0,.8}
\definecolor{ddblue}{rgb}{.0,.0,.4}
\definecolor{dred}{rgb}{.6,.0,.0}
\definecolor{dgreen}{rgb}{.0,.5,.0}
\definecolor{Eeom}{rgb}{.0,.0,.5}

\usepackage[normalem]{ulem}

\usepackage[margin=1in]{geometry}
\allowdisplaybreaks

\newtheorem{lemma}{Lemma}[section]
\newtheorem{theorem}{Theorem}[section]

\newtheorem{corollary}{Corollary}[section]

\theoremstyle{definition}

\theoremstyle{remark}

\newtheorem{remark}{Remark}[section]

\numberwithin{equation}{section}

\usepackage[capitalize,nameinlink]{cleveref}

\crefname{section}{Section}{Sections}
\crefname{subsection}{Subsection}{Subsections}
\crefname{condition}{Condition}{Conditions}
\crefname{hypothesis}{Hypothesis}{Conditions}
\crefname{assumption}{Assumption}{Assumptions}
\crefname{lemma}{Lemma}{Lemmas}
\crefname{claim}{Claim}{Claims}

\makeatletter
\DeclareRobustCommand\widecheck[1]{{\mathpalette\@widecheck{#1}}}
\def\@widecheck#1#2{%
    \setbox\z@\hbox{\m@th$#1#2$}%
    \setbox\tw@\hbox{\m@th$#1%
       \widehat{%
          \vrule\@width\z@\@height\ht\z@
          \vrule\@height\z@\@width\wd\z@}$}%
    \dp\tw@-\ht\z@
    \@tempdima\ht\z@ \advance\@tempdima2\ht\tw@ \divide\@tempdima\thr@@
    \setbox\tw@\hbox{%
       \raise\@tempdima\hbox{\scalebox{1}[-1]{\lower\@tempdima\box
\tw@}}}%
    {\ooalign{\box\tw@ \cr \box\z@}}}
\makeatother

\newcommand{\df}{\coloneqq}
\DeclareMathOperator{\Exp}{\mathbb{E}} 
\DeclareMathOperator{\Prob}{\mathbb{P}} 
\newcommand{\D}{\mathrm{d}}          
\newcommand{\E}{\mathrm{e}}          
\newcommand{\RR}{\mathbb{R}}         
\newcommand{\Rd}{{\mathbb{R}^d}}       
\newcommand{\Ind}{\mathds{1}}            

\newcommand{\uuptau}{\Breve{\uptau}}
\newcommand{\grad}{\nabla}

\newcommand{\sB}{\mathscr{B}}    
\newcommand{\cC}{\pazocal{C}}     
\newcommand{\cD}{\pazocal{D}}     
\newcommand{\sD}{\mathcal{D}}

\newcommand{\sM}{\mathscr{M}}     

\newcommand{\abs}[1]{\lvert#1\rvert}
\newcommand{\norm}[1]{\lVert#1\rVert}
\newcommand{\babs}[1]{\bigl\lvert#1\bigr\rvert}

\DeclareMathOperator*{\esssup}{ess\,sup}

\DeclareMathOperator{\dist}{dist}

\begin{document}

\title[Location of maximizers of eigenfunctions]%
{Location of maximizers of eigenfunctions of fractional Schr\"{o}dinger's equations}

\author[Anup Biswas]{Anup Biswas$^\dag$}
\address{$^\dag$ Department of Mathematics,
Indian Institute of Science Education and Research,
Dr. Homi Bhabha Road, Pune 411008, India}
\email{anup@iiserpune.ac.in}

\date{}

\begin{abstract}
Eigenfunctions of the fractional Schr\"{o}dinger operators in a domain $\cD$ are considered, and a relation between the supremum of the
potential and the distance of a maximizer of the eigenfunction from $\partial\cD$ is established. This, in particular, extends a recent result of
Rachh and Steinerberger \cite{RacSte} to the fractional Schr\"{o}dinger operators. We also propose a fractional version of the Barta's inequality and 
also generalize a celebrated Lieb's theorem for fractional Schr\"{o}dinger operators. As applications of above results we obtain a Faber-Krahn
inequality for non-local Schr\"{o}dinger operators.
\end{abstract}

\keywords{Principal eigenvalue, nodal domain, fractional Laplacian, Barta's inequality, ground state, fractional Faber-Krahn, obstacle problems}

\subjclass[2000]{Primary: 35P15, 35P05, 35B38,}

\maketitle

\section{Introduction}
In this article we are interested in certain properties of eigenfunctions and eigenvalues of fractional Schr\"{o}dinger operators. These operators
have applications in quantum mechanics, for instance, they show up in relativistic Schr\"{o}dinger's equations \cite{CMS, Herbst}. Therefore, understanding
the properties of eigenvalues and eigenfunctions becomes important. Many properties that are known for Laplacian operator have been generalized to
the fractional Laplacian operators. We refer the readers to the recent review by Rupert L. Frank \cite{Frank} (see also \cite[Chapter~4]{Bogdan-etal}) for further motivations
 and recent developments in the study of eigenvalues of fractional Laplacian operators.

Let $X$ be the $d$-dimensional, $d\geq 2$, spherically symmetric $\alpha$-stable process with $\alpha\in (0, 2]$ defined on a complete probability space 
$(\Omega, \mathcal{F}, \Prob)$. In particular, 
$$\Exp\Bigl[\E^{i\xi\cdot (X_t-X_0)}\Bigr]\;=\; \E^{-t\abs{\xi}^\alpha}, \quad \text{for every}\; \xi\in\Rd, \; t\geq 0\,.$$
When $X_0=x$ we say that the $\alpha$-stable process $X$ starts at $x$. As well known for $\alpha=2$, $X$ is the $d$-dimensional Brownian motion, will be denoted by $W$,
running twice fast as the $d$-dimansional standard Brownian motion.
We denote the generator of $X$ by $-(-\Delta)^{\alpha/2}$ where 
\begin{align*}
-(-\Delta)^{\alpha/2}u(x)=\begin{cases}
c_{d, \alpha}\int_{\Rd} \bigl( u(x+y)-u(x)-\Ind_{\{\abs{y}\leq 1\}} y\cdot \grad u(x)\bigr)\frac{1}{\abs{y}^{d+\alpha}} \D{y}, &\text{if~~}\; \alpha\in(0, 2)\,,
\\[5pt]
\Delta u(x), & \text{if~~}\; \alpha=2\,.
\end{cases}
\end{align*}
It is also well known that the sample paths of $X$ are continuous if and only if $\alpha=2$. By $p(t, x, y)$ we denote the transition density of $X_t$ starting
at $x$. Unfortunately, no explicit expression of $p(t, x, y)$ is known unless $\alpha=1, 2$. 
But the following estimate is known for $\alpha\in(0, 2)$: there exists a positive constant $C$ such that
\begin{equation}\label{E1.1}
\frac{1}{C}\, \Bigl(t^{-\nicefrac{d}{\alpha}}\wedge \frac{t}{\abs{x-y}^{d+\alpha}}\Bigr)\;\leq\; p(t, x, y)\;\leq\;
C\, \Bigl(t^{-\nicefrac{d}{\alpha}}\wedge \frac{t}{\abs{x-y}^{d+\alpha}}\Bigr), \; (t, x, y)\in (0, \infty)\times\Rd\times\Rd\,.
\end{equation}
For a set $A\subset\Rd$, we denote the exit time of $X$ from $A$ by $\uptau(A)$ i.e.,
$$\uptau(A)\df \inf\{t>0\;\colon\; X_t\notin A\}\,.$$

Let $\cD$ be an open, connected set in $\Rd$ and $V$ be a bounded function on $\cD$. We are interested in the non-zero
solution $u\in \cC_b(\Rd)$, the space of real-valued bounded, continuous function on $\Rd$, of the following equation
\begin{equation}\label{E1.2}
\begin{aligned}
  -(-\Delta)^{\nicefrac{\alpha}{2}} u + V u &= \; 0\quad \text{in}\;\; \cD\,,\\
  u &=\; 0\quad \text{in}\;\; \cD^c\,. 
\end{aligned}
\end{equation}
For $\alpha=2$ the boundary condition in \eqref{E1.2} should be interpreted as $u=0$ on $\partial\cD$.
Note that we have not imposed any regularity condition on the boundary of $\cD$. However, throughout this article we assume the following without further mention
\begin{itemize}
\item[{\bf(H)}] $u$ has the Feynman-Kac representation in $\cD$, i.e., for any $x\in \cD$ we have
\begin{equation}\label{E1.3}
u(x)\;=\; \Exp_x\Bigl[\E^{\int_0^{t\wedge\uptau(\cD)} V(X_s)\, \D{s}} u(X_{t\wedge \uptau(\cD)})\Bigr]\,, \quad t>0\,.
\end{equation}
\end{itemize}
Here $\Exp_x[\cdot]$ denotes the expectation under the law of the process $X$ starting at $x$. The above hypothesis (H) is not very restrictive and
is satisfied under very mild assumptions on $\cD$ and $V$. For instance, let $V$ be locally H\"{o}lder continuous in $\cD$ and any point $z$ on 
$\partial\cD$ be regular in the sense that $\Prob_z(\uptau(\bar\cD)=0)=1$. The later is known to hold if $\cD$ has an exterior cone property.
On the other hand any viscosity solution of \eqref{E1.2} belongs to $\cC^{\alpha+}(K)$ for any compact set $K$ in $\cD$ (see \cite{RosSer}). Therefore, 
one can safely apply Feynman-Kac formula in any compact set $K\Subset\cD$ and get the representation \eqref{E1.3} with $\uptau(D)$ replaced by
$\uptau(K)$. Now if the boundary is regular we have $\uptau(\cD)=\uptau(\bar\cD)$, and therefore, if we consider a sequence of compact set which
increases to $\cD$ it is easy to obtain \eqref{E1.3} employing a limit argument. Moreover, for $\alpha=2$ the same argument works even when $\partial \cD$ is not regular in the above sense.

Representation \eqref{E1.3} plays a key role in the analysis of this article. Assume that $u\neq 0$ and $x_0\in\cD$ is a maximizer of 
$\abs{u}$ in $\cD$. We may also assume that $u(x_0)>0$, otherwise we multiply $u$ by $-1$. Then from \eqref{E1.3} we obtain that
\begin{align*}
u(x_0)\;=\; \Exp_{x_0}\Bigl[\E^{\int_0^{t} V(X_s)\, \D{s}} u(X_{t})\Ind_{\{\uptau(\cD)> t\}}\Bigr]
\;\leq\; u(x_0)\, \E^{t\, \norm{V}_\infty} \Prob_{x_0}(\uptau(\cD)> t)\, , \quad \forall\; t>0\,,
\end{align*}
where $\norm{V}_\infty\df\sup_{\cD}\abs{V}$. Thus we have
\begin{equation}\label{E1.4}
\E^{t\, \norm{V}_\infty} \Prob_{x_0}(\uptau(\cD)> t)\;\geq \; 1, \quad \text{for every}\; t>0\,.
\end{equation}
The main idea of all the proofs of this article is to exploit the probability in \eqref{E1.4}. Similar expression also appears in \cite[Section~4]{Steiner-14}
where this is used to study certain properties of the nodal domains of the eigenfunctions of Laplacian operator.

To state our first main result we need a few more notations.  We need a class of domains similar to \cite[Theorem~2]{Hayman}.
Let $\sD(d, \beta)$ be the collection of domains in $\Rd$ with the following property: for any
$z\in\partial\cD$ we have $\abs{\cD^c\cap B_{r}(z)}\geq \beta \abs{B_{r}(z)}>0$ for every $r>0$, where $B_{r}(z)$ denotes the ball of radius $r$ around the point $z$.
Note that bounded domains with uniform exterior sphere property belong to $\sD(d, \beta)$ for some $\beta>0$.
Our first main result is the following.

\begin{theorem}\label{T1.1}
Consider $\cD\in\sD(d, \beta)$ for some $\beta\in(0, 1)$. Let $u\neq 0$ satisfy \eqref{E1.2} in the domain $\cD$ and $\abs{u}$ attend its maximum at $x_0\in\cD$. Then there exists a constant $c$, dependent only on $d, \alpha, \beta$, satisfying
\begin{equation}\label{ET1.1A}
\dist(x_0, \partial\cD)\geq c\, \Bigl(\sup_{x\in\cD}\,\abs{V(x)}\Bigr)^{-\nicefrac{1}{\alpha}}\,.
\end{equation}
\end{theorem}
One interesting consequence of Theorem~\ref{T1.1} is the following. By {\it inradius} of a domain $\cD$, denoted by $\mbox{inradius}(\cD)$, we mean the radius of the maximal sphere that can be inscribed inside $\overline\cD$.
\begin{corollary}\label{C1.1}
There exists a universal constant $c$, depending on $d, \alpha$, such that for any bounded, convex set $\cD$ with non-empty interior we have 
\begin{equation*}
\dist(x_0, \partial\cD)\geq c\, \Bigl(\sup_{x\in\cD}\abs{V(x)}\Bigr)^{-\nicefrac{1}{\alpha}}\,,
\end{equation*}
where $x_0$ is a  maximizer of $\abs{u}$ and $u\neq 0$ satisfies \eqref{E1.2}.
 In particular, \eqref{E1.2} does not have a non-zero solution in a bounded, convex domain $\cD$ if
$$\sup_{x\in\cD}\,\abs{V(x)}\;<\; \left[\frac{c}{\mathrm{inradius}(\cD)}\right]^\alpha\,.$$
\end{corollary}

\begin{proof}
Observe that for a convex set $\cD$ we can take $\beta=\frac{1}{2}$, and therefore, the proof follows from Theorem~\ref{T1.1}.
 For the second part, boundedness of $\cD$ assures existence of a maximizer of $\abs{u}$.
\end{proof}

\begin{remark}
The exponent $-\nicefrac{1}{\alpha}$ is optimal in Theorem~\ref{T1.1} and Corollary~\ref{C1.1}. This can be immediately seen as follows. If we let $\cD$ to be $B_r(0)$, ball of radius
$r$ around $0$, then by Corollary~\ref{C1.1} we get 
$$\frac{c}{r} \;\leq \; \lambda^{\nicefrac{1}{\alpha}}_{1, \alpha} (r)\;\leq\; \lambda^{\nicefrac{1}{2}}_{1, 2}(r)\;=\; \frac{\lambda^{\nicefrac{1}{2}}_{1, 2}(1)}{r}\,,$$
where $\lambda_{1, \alpha}(r)$ denotes the principal eigenvalue of $(-\Delta)^{\nicefrac{1}{\alpha}}$ in $B_r(0)$, and the last inequality follows from \cite{Chen-Song}. This shows that the exponent $-\nicefrac{1}{\alpha}$ is optimal.
\end{remark}

Another application of Theorem~\ref{T1.1} gives a generalization to the Barta's inequality \cite{Barta, Sato} for fractional Laplacian in the convex domains.
Classical Barta's inequality provides a bound for the principal eigenvalue of Laplacian in bounded domains. In particular, 
by Barta's inequality one can take the constant
$c$ to be $1$ for $\alpha=2$, in Corollary~\ref{C1.2} below.
\begin{corollary}\label{C1.2}
There exists a universal constant $c$, depending on $d, \alpha$, such that for any bounded, convex domain $\cD$ and $u\in\cC^{\alpha+}_{\mathrm{loc}}(\cD)\cap\cC(\Rd), \, u=0$ on $\cD^c$,
we have
\begin{equation}\label{EC1.2A}
\lambda_{1, \alpha}(\cD)\;\leq c\; \esssup_{x\in\cD}\left|\frac{-(-\Delta)^{\nicefrac{\alpha}{2}} u(x)}{u(x)}\right|\,,
\end{equation}
where $\lambda_{1, \alpha}(\cD)$ denotes the principal eigenvalue of $(-\Delta)^{\nicefrac{\alpha}{2}}$ in $\cD$.
\end{corollary}

For $\alpha=2$  and $d=2$ Rachh and Steinerberger first looked into the estimate
\eqref{ET1.1A} in \cite[Theorem~1]{RacSte}. In particular, they
proved the following.
\begin{theorem}[Rachh \& Steinerberger \cite{RacSte}]\label{T1.2}
Let $\alpha=2$ and $u$ satisfy \eqref{E1.2}. There exists a universal constant $c$ such that for all simply-connected domain $\cD\subset\RR^2$
we have
\begin{equation*}
\dist(x_0, \partial\cD)\geq c\, \Bigl(\sup_{x\in\cD}\,\abs{V(x)}\Bigr)^{-\nicefrac{1}{2}}\,,
\end{equation*}
where $x_0$ is a maximizer of $\abs{u}$.
\end{theorem}
The main idea of the proof again uses \eqref{E1.4}. In \cite{RacSte} the authors give some heuristic to justify a universal upper bound on the probability
$\Prob_{x_0}(\uptau(\cD)>1)$, independent of the domain $\cD$. 
In Section~\ref{S-proof} we shall revisit this proof and give a rigorous argument for this upper bound. Theorem~\ref{T1.2} is a refined
version of a problem of P\'{o}lya and Szeg\"{o} \cite{PolSz}, raised in 1951, which states existence of a universal constant $c_1$ such that for
any simply connected domain $\cD\subset\RR^2$ one has
$$\mbox{inradius}(\cD)\geq c_1\, \lambda_{1,2}(\cD)^{-\nicefrac{1}{2}}\,.$$
Here $\lambda_{1,2}(\cD)$ denotes the principal eigenvalue of the Laplacian in $\cD$ with Dirichlet boundary condition.
This problem of P\'{o}lya and Szeg\"{o} was independently solved
by Hayman \cite{Hayman} and Makai \cite{Makai}. It is also noted by Hayman in \cite{Hayman} that it is not possible to bound the inradius
of $\cD\subset\Rd$ by the principal eigenvalue $\lambda_{1,2}(\cD)$ if $d\geq 3$.
For example, one can remove lines from $\cD$ which decreases the inradius but
does not affect the value of $\lambda_{1,2}(\cD)$. Note that removal of lines does not affect the measure of the domain $\cD$. 
Recall that $\lambda_{1, \alpha}(\cD)$ denotes the principal eigenvalue of $(-\Delta)^{\nicefrac{\alpha}{2}}$ in $\cD$. It is  known from the fractional Faber-Krahn inequality that
$$\lambda_{1, \alpha}(\cD) \;\geq \; C_{d, \alpha}\, \abs{\cD}^{-\nicefrac{\alpha}{d}},$$
where $C_{d, \alpha}$ is a constant and $\abs{\cdot}$ denotes the Lebesgue measure on $\Rd$.
The optimal value of the constant $C_{d, \alpha}$ is attained when $\cD$ is a ball.
Note that if $\lambda_{1, \alpha}(\cD)$ is small then the Lebesgue
measure of $\cD$ has to be large. In fact, $\cD$ has to be {\it fat} in the sense of Lieb \cite{Lieb}. 
More precisely, it is shown in \cite[Corollary~2]{Lieb}
 that for $\alpha=2$ and any $\varepsilon\in(0,1)$,
one can find a constant $r_\varepsilon$, dependent only on $d$, such that for any $\cD\subset\Rd$ there exists a point $x\in\Rd$ satisfying
$$\abs{\cD\cap B_x} \geq \; (1-\varepsilon) \,\abs{B_x},$$
where $B_x$ is the ball of radius $r_\varepsilon \lambda_{1, 2}(\cD)^{-\nicefrac{1}{2}}$ around the point $x$.
By drawing motivation from \cite{Steiner-14}, the above result is refined by
Georgiev \& Mukherjee \cite{GM}
who show that one can choose $x=x_0$ as the centre of $B_x$ where $x_0$ is a maximizer of $\abs{u}$.
Using the ideas of \cite{GM}, Lieb's result is further generalized to bounded potential $V$ by Rachh \& Steinerberger in \cite[Theorem~2]{RacSte}.
Our next result generalizes this Lieb's theorem for Laplace operator to fractional Schr\"{o}dinger operators. Since Lieb's theorem is known for
$\alpha=2$, we only consider $\alpha\in(0, 2)$ for the following result.

\begin{theorem}\label{T1.3}
Let $u$ be a non-zero solution to \eqref{E1.2} and $x_0$ be a maximizer of $\abs{u}$. Let $\norm{V}_\infty$ be positive and $\alpha\in(0, 2)$.
Then for any $\epsilon\in(0,1)$, there exists positive
$r_\epsilon$, dependent on $\varepsilon, d, \alpha$ but not on $\norm{V}_\infty, \cD, x_0$, satsifying 
$$\abs{\cD\cap B_{x_0}}\geq (1-\epsilon)\, \abs{B_{x_0}}\,,$$
where $B_{x_0}$ is the ball around $x_0$ of radius $r_\epsilon \norm{V}_\infty^{-\nicefrac{1}{\alpha}}$.
\end{theorem}
Recently, De Carli \& Hudson \cite{Carli-Hudson} have proved the Faber-Krahn inequality for the eigenfunctions of the Schr\"{o}dinger operator
$-\Delta + V$, and shown that for
a bounded domain $\cD$ with a non-zero solution $u$ satisfying $\Delta u + Vu=0$ in $\cD$, $u=0$ on $\partial\cD$, one has
$$\abs{D}\, \norm{V}^{\nicefrac{d}{2}}_\infty\;\geq \; j^d\, \abs{B_1(0)}\,,$$
where $j=j_{\frac{d}{2}-1}$ is the first zero of the Bessel function $J_{\frac{d}{2}-1}(x)$. It is also shown in \cite{Carli-Hudson} that the constant appearing on
the RHS is sharp. We can generalize this result to the fractional Schr\"{o}dinger operators using Theorem~\ref{T1.3}.

\begin{corollary}
Let $\alpha\in (0,2)$. Suppose that $\cD$ is a bounded domain and $u$ satisfy \eqref{E1.2}. Then there exists a universal constant $c$, depending on $d, \alpha$ but not on $\cD$,
such that
$$\abs{D}\, \norm{V}^{\nicefrac{d}{\alpha}}_\infty\;\geq \; c\,.$$
\end{corollary}

\begin{proof}
Let $x_0$ be a maximizer of $\abs{u}$ in $\cD$ and we let $\varepsilon=\frac{1}{2}$. Let $B$ be the ball of radius $r_\varepsilon \norm{V}_\infty^{-\nicefrac{1}{\alpha}}$ around $x_0$. 
Then by Theorem~\ref{T1.3} we have
$$\abs{\cD}\;\geq\; \abs{\cD\cap B}\;\geq\; \frac{1}{2}\abs{B}=\frac{r^d_\varepsilon\, \abs{B_1(0)}}{2}\norm{V}_\infty^{-\nicefrac{d}{\alpha}}\,.$$
Hence the proof.
\end{proof}

Before we conclude this section let us mention another interesting application of the Feynman-Kac representation of the solution and also an application of Theorem~\ref{T1.1} to the obstacle problems for
fractional Laplacian.

\begin{theorem}\label{T1.4}
There exists a universal constant $c$, depending on $d, \alpha$ but not on $\cD$, such that for any solution $u$ of \eqref{E1.2} we have
\begin{equation}\label{ET1.4A}
\sup_{\cD}\,\abs{u}\;\leq \; c\, \norm{V}^{\nicefrac{d}{2\alpha}}_\infty\,\left[\int_{\cD} \abs{u}^2(y)\, \D{y} \right]^{\nicefrac{1}{2}}\,.
\end{equation}
Moreover, if $\phi_{1, \alpha}$ is the principal eigenfunction of $(-\Delta)^{\nicefrac{\alpha}{2}}$ then for a universal constant $c$, 
depending only on $d, \alpha$,
we have 
\begin{equation}\label{ET1.4B}
\sup_{\cD}\, \phi_{1,\alpha}\;\leq \; c\, \frac{1}{(\mathrm{inradius}(\cD))^{\nicefrac{d}{2}}}\,\left[\int_{\cD} (\phi_{1,\alpha})^2(y)\, \D{y} \right]^{\nicefrac{1}{2}}\,.
\end{equation}
\end{theorem}

Note that the above estimate is a reverse H\"{o}lder inequality for the eigenfunctions. For $\alpha=2$, \eqref{ET1.4B} was obtained by Chiti \cite{Chiti-1, Chiti-2} with an optimal constant $c$. Later 
van den Berg \cite{VanDenBerg} conjectured  existence of a constant $c$, dependent only on $d$, such that for any convex set $\cD\subset\Rd$ with inradius $\rho$ and diameter $D$ one has
$$ \sup_{\cD}\,\abs{u} \;\leq \; c \frac{1}{\rho^{\nicefrac{d}{2}}}\, \Bigl(\frac{\rho}{D}\Bigr)^{\nicefrac{1}{6}}\, \left[\int_{\cD} \abs{u}^2(y)\, \D{y} \right]^{\nicefrac{1}{2}}\,,$$
where $u$ is the principal eigenfunction of $-\Delta$ in $\cD$. Recently, this conjecture is established
in the special case of two dimension by Georgiev, Mukherjee \& Steinerberger \cite{GMS}. Theorem~\ref{T1.4} can be seen as a generalization to
the Chiti's inequality to the eigenfunctions of fractional Laplacian.

\begin{proof}[Proof of Theorem~\ref{T1.4}]
 If $u=0$ there is nothing to prove. Thus, we may assume that $u(x_0)=\sup_{\cD}\abs{u}>0$. It is easy to see from \eqref{E1.3} that
\begin{align*}
\sup_{\cD}\,\abs{u}\;=\; u(x_0) &\leq \; \E^{t\, \norm{V}_\infty} \Exp_{x_0}\left[\abs{u}(X_t) \Ind_{\{t<\uptau\}}\right]\, 
\\[3pt]
&\leq \; \E^{t\, \norm{V}_\infty} \Exp_{x_0}\left[\abs{u}^2(X_t)\right]^{\nicefrac{1}{2}}
\\[3pt]
&\leq \; \E^{t\, \norm{V}_\infty} \left[\int_{\Rd} \abs{u}^2(y) p(t, x_0, y)\, \D{y}\right]^{\nicefrac{1}{2}}, \quad \forall\; t\; >0\,.
\end{align*}
Since $p(t, x, y)\leq C\, t^{-\nicefrac{d}{\alpha}}$ for all $(t, x, y)\in (0, \infty)\times\Rd\times\Rd$, by \eqref{E1.1}, choosing $t=\frac{1}{\norm{V}_\infty}$ above, we obtain \eqref{ET1.4A}.

To prove the second part, we consider the principal eigenpair $(\phi_{1,\alpha}, \lambda_{1, \alpha})$. Therefore, we have $V=\lambda_{1, \alpha}$ in \eqref{ET1.4A}. Then \eqref{ET1.4B} follows from the following estimate
$$\lambda^{\nicefrac{1}{\alpha}}_{1, \alpha}\;\leq\; \lambda^{\nicefrac{1}{2}}_{1, 2} \;\leq\; \lambda^{\nicefrac{1}{2}}_{1, 2}(B_\rho)
\;=\; \frac{1}{\rho}\, \lambda_{1, 2}(B_1(0)),$$
where $\rho=\mathrm{inradius}(\cD)$, $B_\rho$ is a ball of radius $\rho$ inscribed in $\cD$, $\lambda_{1, 2}(B_1(0))$ denotes the first eigenvalue of $-\Delta$ in $B_1(0)$, the first inequality above follows from \cite[Theorem~3.4 and Remark~3.5]{Chen-Song}  and the
second inequality uses monotonicity of principal eigenvalues of Laplacian with respect to the domains. 
\end{proof}

In the rest of this section, we discuss an application of Theorem~\ref{T1.1} in obstacle problems (see \cite{Har-Kro-Paw, GM-1}). To introduce the problem we recall the variational representation of the
principal eigenvalue of the fractional Laplacian operator in a domain $\cD$ i.e., for $\alpha\in(0, 2)$,
$$\lambda_{1, \alpha}(\cD) = \frac{c_{d, \alpha}}{2}\, \inf_{u\in X_0(\cD)} \, \frac{\int_{\Rd\times\Rd}\frac{\abs{u(x)-u(y)}^2}{\abs{x-y}^{d+\alpha}} \D{x}\,\D{y}}{\int_{\cD}\abs{u(x)}^2 \, \D{x}}\,,$$
where 
$$X_0(\cD)\df\{u \in H^{\frac{\alpha}{2}}(\Rd)\;:\; u\neq 0\; \text{and}\; u=0\; \text{in}\; \cD^c\}\,,$$
and $H^{\frac{\alpha}{2}}(\Rd)$ is the fractional Sobolev space.
$\lambda_{1, 2}$ is also defined similarly where the fractional integral in the numerator is replaced by the second moment of $\grad u$ in $\cD$. Let $B_r(x)$ be the ball of radius $r$ around $x$. For a given bounded domain $\cD$ the classical obstacle problem determines the optimal location of the \textit{obstacle} $B_r(x)$ so that the value of $\lambda_{1, \alpha}(\cD\setminus B_r(x))$ is minimized/maximized.
For the case $\alpha=2$ these problems have been studied extensively. See for instance \cite{Har-Kro-Paw, GM-1} and the references therein. Recently, Georgiev \& Mukherjee 
\cite{GM-1} have used Theorem~\ref{T1.1} for the case
$\alpha=2$ to \textit{predict} the location of a convex obstacle $\sB$ which maximizes the value of $\lambda_{1, 2}(\cD\setminus\sB)$. Using Theorem~\ref{T1.1}
 we can generalize their result to any $\alpha\in (0, 2]$.
Fix a domain $\cD\in\sD(d, \beta)$ where $\sD(d, \beta)$ is same as in Theorem~\ref{T1.1}. Let $\phi_{1, \alpha}$ be the principal eigenfunction of $(-\Delta)^{\nicefrac{\alpha}{2}}$ with principal eigenvalue
$\lambda_{1, \alpha}$. Define 
$$\sM\df \{x\in\cD\; :\; \phi_{1, \alpha}(x)=\max_{\cD}\, \phi_{1, \alpha}\}\,.$$
Then the following result follows along the lines of \cite[Theorem~4.1]{GM-1} and using the monotonicity of $\lambda_{1, \alpha}$ with respect to the domain.

\begin{theorem}\label{T1.5}
Let $\cD\in\sD(d, \beta)$ and $\alpha\in(0, 2]$. Then there exists a universal constant $c$, depending on $\alpha, \beta, d$ but not on $\cD$, such that for any bounded, convex set $\sB$ if we have
$\lambda_{1, \alpha}(\cD\setminus (x+\sB))\geq c\ \lambda_{1, \alpha}(\cD)$ for some $x$, then 
$$\max_{z\in \sM}\, \dist(z, x+\sB)\;=\;0\,.$$
\end{theorem}

\section{Proofs of Theorem~\ref{T1.1},\ref{T1.2} and~\ref{T1.3}}\label{S-proof}
We being with the proof of Theorem~\ref{T1.1}
\begin{proof}[Proof of Theorem~\ref{T1.1}]
We assume that $\sup_{x\in\cD}\abs{V(x)}$ is finite, otherwise there is nothing to prove. Without loss of generality
we assume that $u(x_0)>0$, otherwise multiply $u$ by $-1$.  Let us introduce a scaling that would reduce the
problem to the case where $\dist(x_0, \partial\cD)=1$. Note that we may 
assume $x_0=0$, otherwise we translate the domain. Let $R=\dist(0, \partial\cD)$. Defining $w(x)=u(Rx),$ we find that
$$-(-\Delta)^{\nicefrac{\alpha}{2}} w + V_R\, w=0, \quad \text{in}\; \cD_R\df\frac{1}{R}\cD, \quad \text{where}\; V_R(x)=R^\alpha\,V(Rx)\,.$$
This gives us $\dist(0, \cD_R)=1$. Also note that for any $z\in\partial \cD_R$, we have
$$\abs{\cD^c_R\cap B_r(z)}\geq \beta \abs{B_r(z)},\quad \text{for all}\; r>0\,,$$
where $B_r(z)$ is the ball of radius $r$ around $z$. 
Using the above scaled equation in \eqref{E1.4}, we obtain
\begin{equation}\label{ET1.1D}
\E^{t\, \norm{V_R}_\infty} \Prob_{0}(\uptau_R>t)\geq 1, \quad \forall \; t>0\,,
\end{equation}
where $\uptau_R=\uptau(\cD_R)$ is the exit time from the domain $\cD_R$. Now to complete the proof we shall find  $\delta\in(0, 1)$, independent of $R$, such that 
$$\Prob_{0}(\uptau_R> 1)\leq 1-\delta.$$
Then the proof would follow from \eqref{ET1.1D} with $c=\log \frac{1}{1-\delta}$. Let $z_0\in\partial\cD_R$ be such that $\abs{z_0}=1=\dist(0, \partial\cD_R)$.
Let $\sB(z_0)$ be the ball of  radius $\frac{1}{2}$ around $z_0$ and $A=\cD^c_R\cap \sB(z_0)$. By our assertion we have $\abs{A}\geq \beta \abs{\sB(z_0)}$. Since $A\subset \cD^c_R$ we have 
$$\inf_{y\in A}\,\abs{y}\;\geq\; 1, \quad \text{and}\quad \sup_{y\in A}\, \abs{y}\leq 3/2\,.$$
Therefore, it is easy to see from 
\eqref{E1.1} that for $\alpha\in (0, 2)$,
$$\Prob_0(X_1\in A)\; \geq \frac{1}{C}\int_A \frac{1}{\abs{y}^{d+\alpha}}\, \D{y}\geq \frac{1}{C (\nicefrac{3}{2})^{d+\alpha}}\, \abs{A}=
\frac{\beta\,\omega_d\, ({\nicefrac{1}{2}})^d}{C\,2^d\, (\nicefrac{3}{2})^{d+\alpha}}\,,$$
where $\omega_d$ is the volume of the unit ball in $\Rd$. Let us define $\delta=\frac{\beta\,\omega_d}{C2^d 3^{d+\alpha}}$. Thus
$$\Prob_{0}(\uptau_R>1)= 1-\Prob_0(\uptau_R\leq 1)\leq 1- \Prob_0(X_1\in A)\leq \;1-\delta\,.$$
A similar estimate is also possible for $\alpha=2$. This completes the proof.
\end{proof}

\begin{proof}[Proof of Corollary~\ref{C1.2}]
We only consider the case when the supremum on the RHS of \eqref{EC1.2A} is finite. Define
$$V(x)\df \;\frac{-(-\Delta)^{\nicefrac{\alpha}{2}} u(x)}{u(x)}\,, \quad x\in\cD\,.$$
Note that $V$ need not to be defined pointwise. But we can still apply Theorem~\ref{T1.1}. This can be done as follows. For $\delta>0$, we define
$$\cD_{-\delta}=\{y\in\cD\;:\; \dist(y, \partial\cD)<\delta\}, \quad \text{and},\quad \uptau_\delta=\uptau(\cD_{-\delta})\,.$$
Pick a point $x_0\in\cD$ where $\abs{u}$ attends its maximum and $u(x_0)>0$. Consider $\delta\ll\dist(x_0, \partial\cD)$. Now for $0<\varepsilon\ll\delta$,
we consider a mollifier $\varrho_\varepsilon$ supported in $B_\varepsilon(0)$, and define $V_\varepsilon=V\ast \varrho_\varepsilon$. Also define
$$F_\varepsilon(x)\;\df\;-(-\Delta)^{\nicefrac{\alpha}{2}}u(x)-V_\varepsilon(x) u(x)\,.$$ 
Since $V_\varepsilon$
is smooth in $D_{-\delta}$, by Feynman-Kac formula we get  that for any $t>0$
\begin{equation*}
\Exp_{x_0}\Bigl[\E^{\int_0^{t\wedge\uptau_\delta} -V_\varepsilon(X_s)\, \D{s}} u(X_{t\wedge\uptau_\delta})\Bigr] - u(x_0)\;=\;
\Exp_{x_0}\Bigl[\int_0^{t\wedge\uptau_\delta}\E^{\int_0^{s} -V_\varepsilon (X_p)\, \D{p}} F_\varepsilon(X_{s})\D{s}\Bigr]\,,
\end{equation*}
which implies
\begin{align}\label{EC1.2B}
u(x_0) &\;\leq\; \E^{t\,\norm{V}_\infty}\Exp_{x_0}\Bigl[\abs{u}(X_{t\wedge\uptau_\delta})\Bigr] +\,
\E^{t\,\norm{V}_\infty}\, \Exp_{x_0}\Bigl[\int_0^{t} \Ind_{\cD_{-\delta}}(X_s)\abs{F_\varepsilon(X_{s})}\D{s}\Bigr]\nonumber
\\[5pt]
&=\;\E^{t\,\norm{V}_\infty}\Exp_{x_0}\Bigl[\abs{u}(X_{t\wedge\uptau_\delta})\Bigr] +\,
\E^{t\,\norm{V}_\infty}\, \int_0^{t}\int_{\Rd} \Ind_{\cD_{-\delta}}(y)\abs{F_\varepsilon(y)} p(s, x_0, y)\, \D{y}\,\D{s}\,.
\end{align}
Since $F_\varepsilon$ is uniformly bounded in $\cD_{-\delta}$ and $V_\varepsilon\to V$, as $\varepsilon\to 0$, almost everywhere in $\cD_{-\delta}$ 
(and therefore, $F_\varepsilon\to 0$ almost everywhere), 
we get from 
\eqref{EC1.2B} by letting $\varepsilon\to 0$, that
$$u(x_0) \;\leq\; \E^{t\,\norm{V}_\infty}\Exp_{x_0}\Bigl[\abs{u}(X_{t\wedge\uptau_\delta})\Bigr]\,.$$
Now let $\delta\to 0$, and use the fact $\uptau_\delta\to\uptau(\cD)$ almost surely as $\delta\to 0$, to obtain \eqref{E1.4}. Hence we can apply 
the arguments of Theorem~\ref{T1.1} to reach at the same conclusion as in Theorem~\ref{T1.1}.

All the constants involved in the estimate below depend only on $d, \alpha$ and not on the domain. 
\begin{align*}
\esssup_{x\in\cD}\left|\frac{-(-\Delta)^{\nicefrac{\alpha}{2}} u(x)}{u(x)}\right|^{-\nicefrac{1}{\alpha}}&\;\leq\; c\; \dist(x_0, \partial\cD)
\\[1pt]
&\;\leq \; c\;\mbox{inradius} (\cD)
\\[1pt]
&\;\leq\; c_1 \frac{1}{\lambda_{1, 2}(\cD)^{\nicefrac{1}{2}}}
\\[1pt]
&\;\leq\; c_1 \frac{1}{\lambda_{1, \alpha}(\cD)^{\nicefrac{1}{\alpha}}}\,,
\end{align*}
where the first inequality follows from Corollary~\ref{C1.1}, third inequality follows from \cite[Theorem~2]{Hayman} and the fourth inequality follows
from \cite[Theorem~3.4, Remark~3.5]{Chen-Song}. Hence the proof.
\end{proof}

Now we prove Theorem~\ref{T1.2}. Let us again remind the readers  that the proof below is essentially the arguments of \cite{RacSte}. We just give a rigorous proof of the upper bound
estimate of the probability in \eqref{E1.4}.

\begin{proof}[Proof of Theorem~\ref{T1.2}]
As in Theorem~\ref{T1.1}, we assume that $x_0=0$ and $\dist(x_0, \partial \cD_R)=1$ where $\cD_R=\frac{1}{R}\cD$ and $R=\dist(x_0, \cD)$. Let $z_0\in \partial\cD_R$ be such that
$\abs{z_0}=1$. From \eqref{ET1.1D} we obtain
\begin{equation}\label{ET1.2B}
\E^{t\norm{V_R}_\infty} \Prob_{0}(\uptau_R>t)\;\geq\; 1, \quad \forall \; t>0\,,
\end{equation}
where $\norm{V_R}_\infty=R^2\norm{V}_\infty$. Here $\uptau_R$ denotes the exit time of the Brownian motion $W$ from $\cD_R$. 
Fix $t=1$ above. To complete the proof we need to show that for some $\delta\in(0,1)$, independent of $\cD$, we have
\begin{equation}\label{ET1.2C}
\Prob_{0}(\uptau_R>1)\;\leq\; 1-\delta\,.
\end{equation}
Note that the bound in Theorem~\ref{T1.1} depends on $\beta$, and therefore, the arguments of Theorem~\ref{T1.1} is not applicable here. 
We use the property of simply connectedness of $\cD$ to get this estimate.
Without loss of generality we may assume that
$z_0$ lies on the $x$-axis, 
otherwise we rotate the domain around $0$. Consider two concentric circles, centered at $0$, of radii $5/4$ and $7/4$, respectively.
 Let $\varphi:[0, 1]\to\RR^2$ be a $\cC^1$ curve with 
the following properties (see Figure~\ref{Fig-1})
\begin{figure}
\centering
\def\svgwidth{0.3\columnwidth}
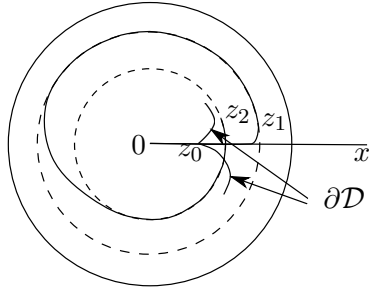
\caption{$d=2$. Dotted lines represent two concentric circles of radii $5/4$ and $7/4$, respectively.}\label{Fig-1}
\end{figure}
\begin{itemize}
\item[(i)] $\varphi(0)=0$, $z_0\in \varphi[0, \frac{1}{16}]$ and $z_2=\varphi(1)$ lies on the circle of radius $5/4$;
\item[(ii)] $\varphi$ intersects itself only at one point, say $\varphi(\frac{1}{4})=\varphi(\frac{3}{4})$, lying on the circle of radius $5/4$;
\item[(iii)] There is an $\varepsilon\in(0, \frac{1}{16})$ such that any curve that belongs to 
$$N_\varepsilon=\{\psi\in \cC([0, 1], \RR^2)\; :\; \sup_{s\in[\frac{1}{8}, \frac{7}{8}]}\abs{\psi(s)-\varphi(s)}<\varepsilon\},$$
gives rise to a closed curve having $z_0$ inside.
\end{itemize}
By Stroock-Varadhan support theorem we know that 
$$\delta\df \Prob_0(W\in N_\varepsilon)\;>\; 0.$$
Note that $\delta$ does not depend on $\cD$. We claim that 
\begin{equation}\label{ET1.2D}
W\in N_\varepsilon\; \Rightarrow \; \uptau_R\;\leq\; 1\,.
\end{equation}
Then \eqref{ET1.2D} implies \eqref{ET1.2C}. Therefore, we only need to establish the above claim. Let $\psi\in N_\varepsilon$. We show that $\psi\cap\cD^c_R\neq\emptyset$ which implies
\eqref{ET1.2D}. Suppose that $\psi\cap\cD^c_R=\emptyset$. We show that there exists a simple closed curve $\xi$ in $\cD$ containing
$z_0$ inside. But this would contradict the simply connectedness of $\cD$. To do so, we define 
$$B_\varepsilon=\left\{x\in \RR^2\; :\; \dist\left(x, \varphi\left[\frac{1}{8}, \frac{7}{8}\right]\right)\;<\;\varepsilon\right\}.$$
We may choose $\varepsilon$ small enough so that $z_0\notin B_\varepsilon$.
 By the property (iii) we can find a closed curve $\gamma:[a, b]\to\cD$
that sits inside  $B_\varepsilon$. This $\gamma$ need not be simple and may have many
small loops inside the tubular part. The required curve $\xi$ can be constructed from $\gamma$ as follows. For every $t\in [a, b]$ we can find $\eta=\eta(t)>0$ and a small disc
$B(\gamma(t))$ around $\gamma(t)$ such that for $s\in (t-\eta, t+\eta)$ we have $\gamma(s)\in B(\gamma(t))\subset \cD\cap B_\varepsilon$. This is possible
as $\gamma$ is continuous and $\cD\cap B_\varepsilon$ is open. Applying compactness of $[a, b]$ we can find $\{t_i\}_{i=0}^k$ such that
$\{(t_i-\eta_i, t_i+\eta_i)\;:\; 0\leq i\leq k\}$ covers $[a, b]$ and $\{B(\gamma(t_i))\; :\; 0\leq i\leq k\}$ forms a chain of overlapping discs in
$\cD\cap B_\varepsilon$. Therefore we can easily construct a piecewise linear closed curve $\tilde\xi$ from $[a, b]$ which is linear, or has
atmost one point of non-differentiability, in each disc $B(\gamma(t_i))$. Since $\tilde\xi$ can have only finitely many small loops, we can erase those loops to form our required simple closed curve $\xi$ in $\cD\cap B_\varepsilon$. This completes the proof.
\end{proof}

Now we remain to prove Theorem~\ref{T1.3}. To do this we need the following lemma.
This result is pretty standard but we include the proof to make this article self-contained.

\begin{lemma}\label{L2.1}
Let $\alpha\in (0, 2)$.
There is a constant $\kappa_1>1$, depending on $\alpha, d$, such that for any $r>0$ and $t>0$ we have
\begin{equation*}
\Prob_{x}(\uptau(B_r(x))\leq t)\;\leq\; \kappa_1 \, t\, r^{-\alpha}\,,
\end{equation*}
for all $x\in\Rd$.
\end{lemma}

\begin{proof}
Note that we may assume $x=0$, without any loss of generality.
Let $f\in\cC^{2}_b(\Rd)$ be such that $f(y)=\abs{y}^{2}$ for $\abs{y}\le \frac{1}{2}$
and $f(y)=1$ for $\abs{y}\ge 1$. Therefore we have a constant $\kappa$ such that
\begin{equation*}
\sup_{y\in \Rd}\;|\grad f(y)|\;\le\; \kappa\,,
\end{equation*}
and for all $y, z \in\Rd$, we have
\begin{equation*}
\babs{f(y+z)-f(z)-\grad f(z)\cdot y}\;\le\; \kappa\, \abs{y}^{2}\,.
\end{equation*} 
Define $f_r(y)=f(\frac{y}{r})$.
Now for $z\in \Bar B_r(0)$ we calculate
\begin{align*}
&\left|\int_{\Rd}
\bigl(f_r(z+y)-f_r(z)-\grad f_r(z)\cdot y\,\bigr)
\frac{1}{\abs{y}^{d+\alpha}}\,\D{y}\right|\\[5pt]
&\;\le\; \left|\int_{|y|\le r}
\bigl(f_r(y+z)-f_r(y)-\grad f_r(z)\cdot y\bigr)
\frac{1}{|y|^{d+\alpha}}\,\D{y}\right|
+\left|\int_{|y|>r}\bigl(f_r(y+z)-f_r(y)\bigr)
\frac{1}{|y|^{d+\alpha}}\,dz\right|\\[5pt]
& \;\le\; \kappa \, \frac{1}{r^{2}}\int_{|y|\le r}|y|^{2-d-\alpha}\,\D{y}
+ 2\int_{|y|>r}|y|^{-d-\alpha}\,\D{y} \\[5pt]
&\;\le\;  \frac{\bar{\kappa}}{r^\alpha}\,,
\end{align*}
for some constant $\bar{\kappa}$. Hence $z\in \Bar B_r(0)$ we have
\begin{equation*}
\left| -(-\Delta)^{\nicefrac{\alpha}{2}} f(z) \right|\;\le\; \frac{c_{d,\alpha}\bar{\kappa}} {r^{\alpha}}\,.
\end{equation*}
Therefore using the It\^o's formula we have for all $t>0$ that
\begin{equation*}
\frac{c_{d,\alpha}\bar{\kappa}}{r^\alpha}\, \Exp_{0}\bigl[\uptau(\Bar B_r(0))\wedge t\bigr]
\;\ge\; \Exp_0\left[\int_0^{\uptau(B_r(0))\wedge t} -(-\Delta)^{\nicefrac{\alpha}{2}} f(X_s)\, \D{s}\right]
\;=\; \Exp_{0}\bigl[f_r(X_{\tau(\Bar B_r(0))\wedge t})\bigr]\,.
\end{equation*}
Since $f_r=1$ on $B^c_r(0)$ we have
$\Prob_{0}(\uptau(\Bar B_r(0))\le t)\le c_{d,\alpha}\bar{\kappa}\,  r^{-\alpha}t$.
The proof follows since $\uptau(\Bar B_r(0))=\uptau(B_r(0))$ almost surely.
\end{proof}

 \begin{figure}[ht]
\centering
\def\svgwidth{0.1\columnwidth}
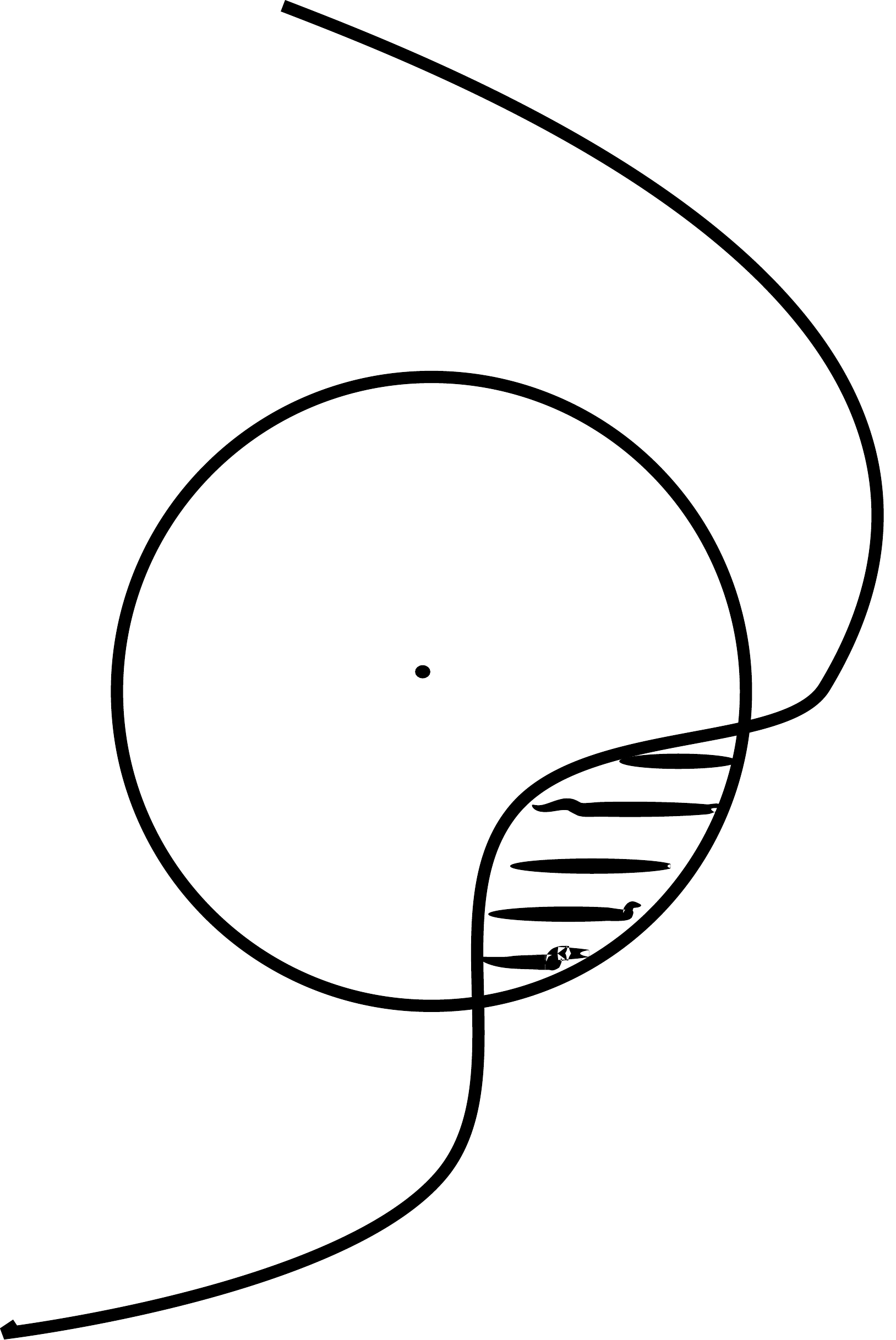
\caption{}\label{Fig-2}
\end{figure}
Now we are ready to prove Theorem~\ref{T1.3}
\begin{proof}[Proof of Theorem~\ref{T1.3}]
As earlier, we assume $u(x_0)>0$. Let $\uptau$ be the exit time of $X$ from $\cD$.
In what follows we denote $\Lambda=\norm{V}_\infty$.
From \eqref{E1.4} we know that
$$\E^{t\,\Lambda} \Prob_{x_0}(\uptau>t)\;\geq\; 1\,,\quad \forall\; t>0\,.$$
Replacing $t$ by $\frac{t}{\Lambda}$ we see that 
\begin{equation}\label{ET1.3A}
\Prob_{x_0}(\uptau>\frac{t}{\Lambda})\;\geq\; e^{-t}\,\quad \forall \; t>0\,.
\end{equation}
By $B_r$ we denote the ball of radius $r$ around $x_0$. 
 Let $E_r=B_r\cap\cD^c$ (see Figure~\ref{Fig-2}).
It is enough to show that for some $r=r_\epsilon\Lambda^{-\nicefrac{1}{\alpha}}$, we have
\begin{equation}\label{ET1.3B}
\abs{E_r}\;\leq\; \epsilon\, \abs{B_{r}}\,.
\end{equation}
We shall use \eqref{ET1.3A} to obtain \eqref{ET1.3B}. Given $\kappa_2\in(0,1/8)$, we can find $t_0$ small enough so that 
\begin{equation}\label{ET1.3C}
\Prob_{x_0}(\uptau\;\leq\;\frac{t_0}{\Lambda} )\leq \kappa_2\,.
\end{equation}
This is possible due to \eqref{ET1.3A}. We shall choose $\kappa_2$ later depending to $\epsilon$. Let $r= t_0^{1/\alpha}\Lambda^{-\nicefrac{1}{\alpha}}$.
By $\uuptau_r=\uptau(E^c_r)$ we denote the hitting time to $E_r$ by $X$. Denote also by $\uptau_r=\uptau(B_r(x_0))$.
We use L\'{e}vy system formula (see for instance, \cite[Proposition~2.3 and Remark~2.4]{Bass-Levin}) to estimate the following probability
\begin{align*}
\Prob_{x_0}(\uuptau_r\leq \uptau_r\wedge t) &=\; \Exp_{x_0}\Bigl[\sum_{s\leq \uuptau_r\wedge \uptau_r\wedge t} \Ind_{\{X_{s-}\neq X_s,\; X_s\in E_r\}}\Bigr]
\\
&=\;c_{d,\alpha} \Exp_{x_0}\Bigl[\int_0^{\uuptau_r\wedge \uptau_r\wedge t} \int_{E_r} \frac{\D{z}}{\abs{X_s-z}^{d+\alpha}}\Bigr]\quad [\text{By L\'{e}vy system formula}]
\\
&\geq\; \frac{c_{d,\alpha}}{2^{d+\alpha}} \Exp_{x_0}\Bigl[\int_0^{\uuptau_r\wedge \uptau_r\wedge t} \frac{\abs{E_r}}{\abs{r}^{d+\alpha}}\Bigr]
\\
&=\; \kappa_3 r^{-\alpha} \frac{\abs{E_r}}{\abs{B_r}} \Exp_{x_0}[\uuptau_r\wedge \uptau_r\wedge t]\,,
\end{align*}
for some constant $\kappa_3$, dependent only on $d, \alpha$.
Now let us find a lower bound for the expectation above. By our choice of $r$ above we have $t r^{-\alpha}=t \frac{\Lambda}{t_0}$. Therefore by choosing
$\hat{t}=\frac{t_0}{4\kappa_1\Lambda}< \frac{t_0}{\Lambda}$ we obtain from Lemma~\ref{L2.1} that
\begin{equation}\label{ET1.3D}
\Prob_{x_0}(\uptau_r \leq \hat{t})\leq \frac{1}{4}\,.
\end{equation}
Fix this choice of $\hat{t}$ which depends on $t_0, \Lambda$. It is also easy to see that
$$\Prob_{x_0}(\uuptau_r \leq \uptau_r\wedge \hat{t})\;\leq\;  \Prob_{x_0}(\uuptau_r\leq  \hat{t})
\;\leq\; \Prob_{x_0}(\uptau\leq  \hat{t})\;\leq\; \Prob_{x_0}(\uptau\leq  \frac{t_0}{\Lambda})\;\leq\; \kappa_2,$$
where the last estimate follows from \eqref{ET1.3C}. Since $\kappa_2\leq 1/8$ we get 
\begin{equation}\label{ET1.3E}
\Prob_{x_0}(\uuptau_r\leq \uptau_r\wedge \hat{t})\;\leq\; 1/8\,.
\end{equation}
Therefore combining \eqref{ET1.3D} and \eqref{ET1.3E} we have
\begin{align*}
\Exp_{x_0}[\uuptau_r\wedge \uptau_r\wedge \hat{t}]&\;\geq\; \hat{t}\, \Prob_{x_0} (\uuptau_r\wedge \uptau_r>\hat{t})
\\
&\geq\; \hat{t}\, [1-\Prob_{x_0} (\uuptau_r\leq \uptau_r\wedge \hat{t})-\Prob_{x_0} ( \uptau_r\leq \hat{t})]
\\
& \geq\; \frac{1}{2}\hat{t}\,.
\end{align*}
Since $\hat{t} r^{-\alpha}=\frac{1}{4\kappa_1}$, we obtain from the above estimates that 
$$\frac{\abs{E_r}}{\abs{B_r}}\;\leq\; \frac{8\kappa_1}{\kappa_3}\Prob_{x_0} (\uuptau_r\leq \uptau_r\wedge \hat{t})\;\leq\; \frac{8\kappa_1}{\kappa_3}\kappa_2\,.$$
where the last estimate is also obtained as before. Now we choose $\kappa_2$ depending on $\varepsilon$ so that 
$$\frac{\abs{E_r}}{\abs{B_r}}\;\leq\; \varepsilon\,,$$
and this gives us $t_0$ and 
$r_\varepsilon= t_0^{1/\alpha}$.

\end{proof}

\section*{Acknowledgements}
Special thanks to my colleague Tejas Kalelkar for teaching me Inkscape which has been used to draw the diagrams of this article. The author is
indebted to Stefan Steinerberger for constructive comments and suggestions.
The research of Anup Biswas was supported in part by an INSPIRE faculty
fellowship and DST-SERB grant EMR/2016/004810.

\bibliographystyle{plain}      

\bibliography{Eigen}
\end{document}